\documentclass[reqno,12pt]{amsproc}
\usepackage{ucs}
\usepackage[english]{babel}
\usepackage{amssymb,amsmath}
\usepackage{indentfirst}
\usepackage[usenames]{color}
\usepackage{graphicx}
\usepackage[left=2.5cm,right=2.5cm,top=2cm,bottom=2cm,bindingoffset=0cm]{geometry}
\usepackage{amsthm}
\usepackage{amsfonts}
\usepackage{graphicx}
\graphicspath{{pictures/}}
\DeclareGraphicsExtensions{.pdf,.png,.jpg}
\usepackage{wrapfig}

\theoremstyle{plain}  
\newtheorem*{acknowledgements}{Acknowledgements}
\newtheorem{theorem}{\bf Theorem}
\newtheorem{proposition}{\bf Proposition}
\newtheorem{lemma}{\bf Lemma}
\newtheorem{corollary}{\bf Corollary}

\newtheorem{conjecture}{\bf Conjecture}

\newtheorem{remark}{\bf Remark}

\renewcommand{\Re}{\mathop{\mathrm{Re}}\nolimits}

\renewcommand{\a}{\alpha}

\newcommand{\sinc}{\operatorname{sinc}}
\newcommand{\supp}{\operatorname{supp}}
\newcommand{\dist}{\operatorname{dist}}
\newcommand{\Si}{\operatorname{Si}}

\usepackage{xcolor}
\usepackage{hyperref}

\usepackage{float}

\hypersetup{pdfstartview=FitH, linkcolor=linkcolor,urlcolor=urlcolor, colorlinks=true, linkcolor=blue, citecolor=blue}

\begin{document} 
\title{Quadratic Spectral Concentration \\ of Characteristic Functions}
\author{Kristina Oganesyan}
\address{Lomonosov Moscow State University, Moscow Center for fundamental and applied mathematics, Centre de Recerca Matem\`atica, Universitat Aut\`onoma de Barcelona}
\email{oganchris@gmail.com}
\thanks{This research was supported by the Russian Science Foundation (project no. 24-11-00114).}
\date{}

\begin{abstract}
It is known that the inequality 
\begin{align*}
\int_{-W/2}^{W/2}|\widehat{f}(\xi)|^2d\xi\leq \int_{-W/2}^{W/2}|\widehat{|f|^*}(\xi)|^2d\xi
\end{align*}
between the quadratic spectral concentration of a function and that of its decreasing rearrangement holds for any function $f\in L^2,\;|\supp f|=T,$ if and only if the product $WT$ does not exceed the critical value $\approx 0.884$. We show that by restricting ourselves to characteristic functions we can enlarge this range up to $WT\leq 4/3$. Besides, we establish various properties of minimizers of the difference $\int_{-W/2}^{W/2}|\widehat{\chi_A^*}(\xi)|^2d\xi-\int_{-W/2}^{W/2}|\widehat{\chi_A}(\xi)|^2d\xi$ over sets $A$ of finite measure and prove that this difference is non-negative for all $W,T>0$ if $A$ is the union of two intervals. As a corollary, we obtain a sharp (up to a constant) estimate for the $L_2$-norms of non-harmonic trigonometric polynomials with alternating coefficients $\pm 1$.
\end{abstract}
\maketitle

\section{Introduction}
An intriguing question concerning the quadratic spectral concentration of a function was posed by D.L. Donoho and Ph.B. Stark in \cite{DS}: is it true that for any positive numbers $T$ and $W$, among all functions $f\in L^2(\mathbb{R})$ with $\|f\|_2=1$ and $|\supp f|\leq T$, the maximum of $\int_{-W/2}^{W/2}|\widehat{f}(\xi)|^2d\xi$ is attained for an $f$ supported on an interval? They conjectured that the answer to this question should be positive and supported this conjecture by the proof for the case of
$$WT\leq T_0:=\min\{x>0:\sinc x=\max_{y\geq 1}\sinc y\}\approx 0.884,$$
where $\sinc x=\sin(\pi x)/(\pi x)$ and the Fourier transform is defined by $\widehat{f}(\xi)=\int_{\mathbb{R}}f(x)e^{-2\pi i x\xi}dx$. This was an immediate consequence of the inequality 
\begin{align}\label{abs1}
\int_{-W/2}^{W/2}|\widehat{f}(\xi)|^2d\xi\leq \int_{-W/2}^{W/2}|\widehat{|f|^*}(\xi)|^2d\xi
\end{align}
which, as they showed, is valid for $WT\leq T_0$. In other words, the core of their result was proving that rearrangement increases the quadratic spectral concentration of a function whenever the frequencies are low enough. Although inequality \eqref{abs1} does not hold in general once $WT>T_0$ (see Remark \ref{rem1} below), we believe that it should be true for any $W,T>0$ if we restrict ourselves to functions assuming only the values $0$ and $1$. This gives rise to
\begin{conjecture}\label{conj_1} For any measurable set $A,\;|A|<\infty$, and any $W>0$,
\begin{align*}
\int_{-W/2}^{W/2}|\widehat{\chi_A}(\xi)|^2d\xi\leq \int_{-W/2}^{W/2}|\widehat{\chi_A^*}(\xi)|^2d\xi.
\end{align*}
\end{conjecture}
Note that relations between norms of functions and those of their rearrangements (cf. \eqref{abs1}) with some constant factors on the right-hand side are well studied for much more general cases (see e.g. \cite{J,Mo}). 
 
Importantly, Conjecture \ref{conj_1} can be equivalently rewriten (see Corollary \ref{cor} and Propositions \ref{form} and \ref{equal}) in terms of weighted $L^2$-norms of non-harmonic trigonometric polynomials as follows:
\begin{conjecture}\label{conj_2}
For any positive integer $n$ and any real $\alpha_1\leq \alpha_2\leq ...\leq \alpha_{2n}$, one has
\begin{align*}
\int_1^{\infty}\frac{\Big|\sum_{j=1}^{2n}(-1)^je^{i\alpha_j t}\Big|^2}{t^2}dt\geq \int_1^{\infty}\frac{\Big|1-e^{i\sum_{j=1}^n(\alpha_{2j}-\alpha_{2j-1})t}\Big|^2}{t^2}dt.
\end{align*}
\end{conjecture}

Unable to prove Conjecture \ref{conj_1} in its generality, we can nevertheless go beyond the critical value $T_0$ when dealing with characteristic functions.

\begin{theorem}\label{thm} For any positive $T$ and $W$ with $WT\leq 4/3$, for any measurable set $A$, $|A|=T$, there holds
\begin{align*}
\int_{-W/2}^{W/2}|\widehat{\chi_A}(\xi)|^2d\xi\leq \int_{-W/2}^{W/2}|\widehat{\chi_A^*}(\xi)|^2d\xi.
\end{align*}
\end{theorem}

\begin{remark}\label{rem1}
If we allow a function $f$ to assume at least two different positive values, the inequality above is no longer true for $f$ in place of $\chi_A$ and $WT>T_0$ (while in case of $WT\leq T_0$ it is true according to the proof of \cite[Th. 2]{DS}). More precisely, if $WT>T_0$, there exists a function $f_T$ with $|\supp f_T|\leq T$ assuming two different positive values such that \eqref{abs1} does not hold with $W=1$.
\end{remark}

In what follows, for two functions $f_1$ and $f_2$ and a fixed $W>0$, we define
\begin{align*}
(f_1,f_2):=\int_{-\infty}^{\infty}\int_{-\infty}^{\infty}f_1(t)f_2(s)W\sinc (W(s-t)) dsdt,
\end{align*}
if this integral exists, and with a slight abuse of notation, for two measurable sets $M_1$ and $M_2$, we write $(M_1,M_2):=(\chi_{M_1},\chi_{M_2})$. Then
\begin{align*}
\int_{-W/2}^{W/2}|\widehat{f}(\xi)|^2d\xi&=\int_{-\infty}^{\infty}\int_{-\infty}^{\infty}f(t)f(s)W\sinc (W(s-t)) dsdt=(f,f).
\end{align*}
Observe also that by means of dilation we can assume that $W=1$ and $T\leq 4/3$ in order to prove Theorem \ref{thm}.

The paper is organized as follows. First, in Section \ref{sec_basic}, we show that the maximum of $\int_{-W/2}^{W/2}|\widehat{\chi_A}(\xi)|^2d\xi,\;|A|\leq T$, for a fixed $T$ is attained at a set $A$ being a finite union of intervals and that, moreover, such $A$ must be a superlevel set for the function $\int_{A}\sinc (x-t)dt$. Further, in Section \ref{sec_general}, we obtain formulas for $\int_{-W/2}^{W/2}|\widehat{\chi_A^*}(\xi)|^2d\xi-\int_{-W/2}^{W/2}|\widehat{\chi_A}(\xi)|^2d\xi$ in case where $A$ is a finite union of intervals and make several useful preliminary observations. Section \ref{sec_pf} contains the proofs of Theorem \ref{thm} and Remark \ref{rem1}. In Section \ref{sec_two}, we show that if $A$ is the union of two intervals, then bringing these intervals together increases the norm $\int_{-W/2}^{W/2}|\widehat{\chi_A}(\xi)|^2d\xi$, which settles Conjecture \ref{conj_2} for $n=2$. In Section \ref{sec_particular}, we treat some particular cases in which Conjecture \ref{conj_2} can be easily proved. Next, using the results obtained in the previous sections, we establish in Section \ref{sec_l2} the following estimate.
   \begin{theorem}\label{l2} For any positive integer $n$ and any real $\alpha_1\leq \alpha_2\leq ...\leq \alpha_{2n}$,
\begin{align*}
\int_{-1/2}^{1/2}\Big|\sum_{j=1}^{2n}(-1)^je^{i \a_jt}\Big|^2dt<\min\Big\{54n-47,\frac{\pi}{2} \sum_{j=1}^n (\a_{2j}-\a_{2j-1})\Big\}.
\end{align*}
   \end{theorem}
We note that the linear in $n$ bound in Theorem \ref{l2} is up to a constant sharp, since it is attained for $\alpha_1,\alpha_2,...,\alpha_{2n}$ with sufficiently large gaps between them by the Ingham inequality (see \cite{I}, and we also refer the reader to \cite{JS} and references therein for estimates of similar norms in case of non-harmonic trigonometric polynomials with general coefficients). Theorem \ref{l2} motivates as well an additional observation made in Section \ref{sec_obs}, which roughly speaking says the following: if $n$ is the minimal such that for a function $f=\chi_A$ with $A$ being the union of $n$ intervals, inequality \eqref{abs1} is violated, then the minimum of $\int_{-W/2}^{W/2}|\widehat{\chi_A^*}(\xi)|^2d\xi-\int_{-W/2}^{W/2}|\widehat{\chi_A}(\xi)|^2d\xi$ over all such $A$ is achieved on a set of sufficiently large measure.  

\section{Basic properties of extremizers}\label{sec_basic}

\begin{proposition}\label{p} Let a set $A,\;|A|\leq T,$ deliver the maximum of
\begin{align*}
\max_{\text{measurable}\;S,\;|S|\leq T}\int_{-W/2}^{W/2}|\widehat{\chi_{S}}(\xi)|^2d\xi.
\end{align*}
Then there exists a constant $c>0$ such that $A$ is up to a set of measure zero equal to
\begin{align*}
\Big\{x:\int_{A}\sinc (x-t)dt\geq c\Big\}.
\end{align*}
\end{proposition}

\begin{proof}
First, let us prove the statement for some $c$. Assume the contrary. Then, by the continuity of the measure, there exist $\varepsilon,\delta>0$ and two sets $A_0\subset A$ and $B_0\cap A=\varnothing$ of measure $\varepsilon$ such that for any $x\in A_0,y\in B_0$ we have $\int_{A}\sinc (y-t)dt-\int_{A}\sinc (x-t)dt\geq \delta$. However, by the choice of $A$,
\begin{align*}
0&\leq (A,A)-(B_0\cup A\setminus A_0,B_0\cup A\setminus A_0)\\
&=-(A_0,A_0)-(B_0,B_0)+2(B_0,A_0)+2(A_0-B_0,A)\leq -2\varepsilon\delta+\mathcal{O}(\varepsilon^2),
\end{align*}
which is a contradiction. 

Let us show now that $c>0$. If $c<0$, then by a similar argument, we can find $\delta>0$ and a set $A_0\subset A$ of measure $\varepsilon>0$ with $(A,A_0)\leq -\delta$ and thus,
\begin{align*}
0&\leq (A,A)-(A\setminus A_0,A\setminus A_0)=-(A_0,A_0)+2(A_0,A)\leq -2\varepsilon\delta+\mathcal{O}(\varepsilon^2).
\end{align*}
So, the only possible case is $c=0$. But then, as $\widehat{\chi_A}$ cannot be compactly supported,
\begin{align*}
|A|&=\int_{\mathbb{R}}\int_A\sinc(x-y)dxdy\leq \int_A\int_A\sinc(x-y)dx dy\\
&=\int_{-1/2}^{1/2}|\widehat{\chi_A}(\xi)|^2<\|\widehat{\chi_A}\|_2=\|\chi_A\|_2=|A|,
\end{align*}
the contradiction concluding the proof. 
\end{proof}

\begin{corollary}\label{cor} Let a set $A,\;|A|\leq T$, deliver the maximum of
\begin{align*}
\max_{\text{measurable}\;S,|S|\leq T}\int_{-W/2}^{W/2}|\widehat{\chi_{S}}(\xi)|^2d\xi.
\end{align*}
Then, up to a set of measure zero, $A$ is a finite union of intervals.
\end{corollary}

\begin{proof}
According to Proposition \ref{p}, we can assume that $A=\{x:\int_{A}\sinc (x-t)dt\geq c\}$ for some $c>0$. Suppose that $A$ is not bounded. Take some $\varepsilon>0$ and find a bounded $A'\subset A$ with $|A'|\geq |A|-\varepsilon$. Then for a ponit $x,\;\dist(a,A')\geq X$, we have $\int_{A}\sinc(x-t)dt\leq |A|/X+\varepsilon<c$ for large enough $X$ and small enough $\varepsilon>0$. This contradiction shows that $A$ must be bounded. If there are infinitely many boundary points in $A$, then the function 
\begin{align}\label{FA_x}
F_A(x):=\int_{A}\sinc (x-t)dt
\end{align}
being an analytic function assumes the value $c$ infinitely many times in $A$ and must be therefore equal to $c>0$ everywhere. However, $F_{A}(x)\to 0$ as $x\to\infty,\; x\in\mathbb{R}$, and hence, $A$ is a finite union of intervals.
\end{proof}

\section{Reformulations of the problem}\label{sec_general}
One of the most exciting features of the problem we are dealing with is its multifacetedness in the sense that one can approach it from diverse perspectives, which have essentially different nature and involve Fourier analysis, complex analysis, and number theory. In this section, we will look at our problem from different sides and prepare the ground for our further analysis.

Due to the conclusion of Section \ref{sec_basic}, from now on we only have to consider characteristic functions of finite unions of intervals: 
\begin{align}\label{J}
J:=\cup_{j=1}^{n}[x_{2j-1},x_{2j}],\quad x_{j+1}-x_{j}=:a_j,\;j=1,2,...,2n-1.
\end{align}
We emphasize that the definition of the set $J$ clearly depends on the choice of $\{x_j\}$ or, equivalently, of $\{a_j\}$ (the latter is up to translations of $J$, which for us does not make any difference). So once we have either of these sequences fixed, we can define the other sequence and the set $J$. For the sake of convenience, we will omit the dependence of $J$ on any of these sequences: it will be always clear from the context by which sequence $J$ is generated.
 
First of all, we should find an expression for $(J,J)$, which essentially can be obtained by noting that
\begin{align*}
([0,a],[0,a])=\int_0^a \Si(x)+\Si(a-x)dx=2x\Si(x)-\frac{4}{\pi^2}\sin^2\frac{\pi x}{2}=2a\int_0^{a/2} \sinc^2 x dx,
\end{align*}
where $\Si(x)=\int_0^{x}\sinc t dt$, and applying the inclusion-exclusion principle. We will proceed by induction the base being

\begin{lemma}\label{prop_1} For two intervals $I_1$ and $I_2$ with $|I_1|=a,\;|I_2|=b$ and $\dist(I_1,I_2)=\varepsilon$, there holds
\begin{align*}
(I,I)-(I_1\cup I_2,I_1\cup I_2)=\frac{2}{\pi^2}\int_0^1t^{-2}\Re (e^{i\pi a t}-1)(e^{i\pi b t}-1)(e^{i\pi\varepsilon t}-1)dt,
\end{align*}
where $I:=[0,a+b]$.
\end{lemma}

\begin{proof}
We have
\begin{align*}
(I,I)-(I_1\cup I_2,I_1\cup I_2)&=2([-b,0],[0,a])-2(I_1,I_2)\\
&=2\int_0^a\int_0^b\sinc(x+y)-\sinc(\varepsilon+x+y)dxdy\\
&=2\int_0^a\int_0^b\int_0^1(\cos\pi(x+y)t-\cos\pi(\varepsilon+x+y)t)dtdxdy\\
&=\frac{8}{\pi^2}\int_0^1 t^{-2}\sin\frac{\pi at}{2}\sin\frac{\pi bt}{2}\Big(\cos\frac{\pi(a+b)t}{2}-\cos\frac{\pi(2\varepsilon+a+b)t}{2}\Big)dt\\
&=\frac{16}{\pi^2}\int_0^1t^{-2}\sin\frac{\pi at}{2}\sin\frac{\pi\beta t}{2}\sin\frac{\pi \varepsilon t}{2}\sin\frac{\pi(a+b+\varepsilon)t}{2}dt\\
&=\frac{2}{\pi^2}\int_0^1t^{-2}\Re (e^{i\pi a t}-1)(e^{i\pi b t}-1)(e^{i\pi\varepsilon t}-1)dt,
\end{align*}
and the needed is proved.
\end{proof}

\begin{remark} Note that the symmetry in $a,b,\varepsilon$ of the expression obtained in Lemma \ref{prop_1} is not surprising. Indeed, if, say, $a>\varepsilon$, divide $I_1$ into two subintervals $I_1'$ and $I_1''$ of lengths $\varepsilon$ and $a-\varepsilon$. Then to glue together $I_1,I_2$ is the same as to put $I_1'$ into the $``$hole$"$ between $I_1$ and $I_2$ keeping $I_1''$ at the same place, which corresponds to the situation of two intervals $I_1'$ and $I_2$ of lenghts $\varepsilon$ and $b$ and the hole of length $a$ between them.
\end{remark}

For a sequence of non-negative numbers $\{c_j\}_{j=1}^{2n-1}$ and a function $\varphi$, define
\begin{align}\label{Top}
T(\{c_j\},\varphi):=\sum_{s=1}^{2n-1}(-1)^s\sum_{i=1}^{2n-s}\varphi(c_i+c_{i+1}...+c_{i+s-1})+\varphi(c_1+c_3+...+c_{2n-1})
\end{align}
and denote by $T_k(\{c_j\},\varphi)$ the part of $T(\{c_j\},\varphi)$ that involves $c_k$. Note that defining $y_j:=\sum_{k=1}^{j-1}c_j,\;j=1,2,...,2n$, we have
\begin{align}\label{helpful}
T(\{c_j\},\varphi)=\sum_{k>s}(-1)^{k-s}\varphi(y_k-y_s)+\varphi\Big(\sum_{j=1}^n (y_{2j}-y_{2j-1})\Big).
\end{align}

\begin{proposition}
For any $n\in\mathbb{N}$ and any sequence $\{c_j\}_{j=1}^{2n-1}$, there hold
\begin{align}\label{t_zero}
T(\{c_j\},x)=0=T(\{c_j\},x^2)
\end{align}
and 
\begin{align}\label{t_j_zero}
T_k(\{c_j\},x)=0,\;k=1,2,...,2n-1.
\end{align}
\end{proposition}

\begin{proof}
Note that \eqref{t_j_zero} implies the first equality in \eqref{t_zero}, as \eqref{t_j_zero} shows that there is no contribution of $c_k$ in $T(\{c_j\},x)$. So, let us prove \eqref{t_j_zero}, which is equivalent to
\begin{align}\label{show}
\sum_{s=1}^{2n-1}(-1)^{s}\sum_{i=\max\{1,k-s+1\}}^{\min\{k,2n-s\}}(c_i+c_{i+1}...+c_{i+s-1})+\chi_{\{k\;\text{is odd\}}}(c_1+c_3+...+c_{2n-1})=0.
\end{align}
To this end, denote by $C(m,N)$ the coefficient at $c_m$ in the following more general expression:
\begin{align}\label{show2}
\sum_{s=1}^{N}(-1)^{s}\sum_{i=\max\{1,m-s+1\}}^{\min\{m,N-s+1\}}(c_i+c_{i+1}...+c_{i+s-1})+\chi_{\{N\;\text{is odd}\}}(c_1+c_3+...+c_{N}).
\end{align}
We have
\begin{align*}
C(m,N)&=\sum_{s=1}^{m}(-1)^s\sum_{i=m-s+1}^{m}1+\sum_{s=m+1}^{N-m}(-1)^{s}\sum_{i=1}^{m}1+\sum_{N-m+1}^{N}(-1)^s\sum_{i=1}^{N-s+1}1+\chi_{\{m\;\text{is odd}\}}\chi_{\{N\;\text{is odd}\}}\\
&=\sum_{s=1}^{m}s(-1)^s+m\sum_{s=m+1}^{N-m}(-1)^{s}+\sum_{N-m+1}^{N}(N-s+1)(-1)^s+\chi_{\{m\;\text{is odd}\}}\chi_{\{N\;\text{is odd}\}}\\
&=\sum_{s=1}^{m}s((-1)^{N+1}+1)(-1)^s+m\sum_{s=m+1}^{N-m}(-1)^{s}+\chi_{\{m\;\text{is odd}\}}\chi_{\{N\;\text{is odd}\}}=0,
\end{align*}
whence, in particular, the coefficient at $c_k$ on the left-hand side of \eqref{show} is $0$. Further, fix some $\ell$ and let $t:=|\ell-(2k-1)|$. We can observe then that the coefficient $\tilde{C}(\ell,n)$ at $c_{\ell}$ on the left-hand side of \eqref{show} is equal to $C(\min\{\ell,2k-1\},2n-1-t)=0$, which completes the proof of \eqref{t_j_zero}.

The equality $T(\{c_j\},x^2)=0$ is to be shown in the same fashion. The terms of the form $c_j^2$ in $T(\{c_j\},x^2)$ disappear by the identity $T(\{c_j\},x)=0$, while the coefficients at $c_jc_k,\;j\neq k,$ are zero by a similar argument.
\end{proof}

After all, we are in a position to obtain a general expression for $(I,I)-(J,J)$, where from now on by $I$ we mean the interval $[0,\sum_{j=1}^n a_{2j-1}]$, $J$ is as in \eqref{J}, and $T$ stands for the length of $I$, i.e. $T=|I|=|J|$ (we assure the reader that there will be no ambiguity caused by the latter notation and that of the operator introduced in \eqref{Top}). 
\begin{proposition}\label{form} There holds
\begin{align}\label{form_eq}
(I,I)-(J,J)&=\frac{2}{\pi^2}\int_0^1 \frac{T(\{a_j\},1-\cos \pi xt)}{t^2}dt=4T(\{a_j\},G(x/2))\nonumber\\
&=-\frac{1}{\pi^2}\int_0^1\frac{\Big|\sum_{j=1}^{2n}(-1)^je^{i\pi x_jt}\Big|^2}{t^2}dt+\frac{1}{\pi^2}\int_0^1\frac{\Big|1-e^{i\pi \sum_{j=1}^n(x_{2j}-x_{2j-1})t}\Big|^2}{t^2}dt,
\end{align}
where $G(x):=x\int_0^x\sinc^2 t dt=\int_0^x\Si(2y)dy$.
\end{proposition}

\begin{proof} First, observe that 
\begin{align*}
\frac{\pi^2}{2} &((I,I)-(J,J))\\
&=\Re\int_0^1t^{-2}\Big(\sum_{s=1}^{2n-1}(-1)^{s-1}\sum_{i=1}^{2n-s}e^{i\pi (a_i+a_{i+1}...+a_{i+s-1})t}\Big)-e^{i\pi (a_1+a_3+...+a_{2n-1})t}-(n-1)dt.
\end{align*}

Let us prove the statement by induction. The case $n=2$ follows directly from Lemma \ref{prop_1}. Moreover, from Lemma \ref{prop_1} we derive that 
\begin{align*}
([0,a]\cup [a+\delta,a+\delta+b])&-([0,a]\cup [a+\delta+\varepsilon,a+\delta+\varepsilon+b])\\
&=\frac{1}{\pi^2}\int_0^1t^{-2}e^{i\pi\delta t}(e^{i\pi at}-1)(e^{i\pi bt}-1)(e^{i\pi\varepsilon t}-1)dt,
\end{align*}
so the induction step from $\{a_k\}_{k=3}^{2n-1}$ to $\{a_k\}_{k=1}^{2n-1}$ then is provided by 
\begin{align*}
(e^{i\pi a_1 t}&-1)(e^{i\pi a_{2k+1} t}-1)\sum_{k=1}^{n-1}\exp\Big(i\pi t\sum_{l=2}^ka_{2l-1}\Big)\Big(\exp\Big(i\pi t\sum_{l=1}^ka_{2l}\Big)-1\Big)\\
&=(e^{i\pi a_1 t}-1)\sum_{k=1}^{n-1}\Big(\exp\Big(i\pi t\sum_{l=2}^{2k+1}a_l\Big)-\exp\Big(i\pi t\sum_{l=1}^{k}a_{2l+1}\Big)\\
&\qquad\qquad-\exp\Big(i\pi t\sum_{l=2}^{2k}a_l\Big)+\exp\Big(i\pi t\sum_{l=1}^{k-1}a_{2l+1}\Big)\Big)\\
&=(e^{i\pi a_1 t}-1)\Big(-\exp\Big(i\pi t\sum_{l=1}^{n-1}a_{2l+1}\Big)+1+\sum_{k=1}^{n-1}\Big(\exp\Big(i\pi t\sum_{l=2}^{2k+1}a_l\Big)-\exp\Big(i\pi t\sum_{l=2}^{2k}a_l\Big)\Big)\\
&=-\exp\Big(i\pi t\sum_{l=0}^{n-1}a_{2l+1}\Big)+\exp\Big(i\pi t\sum_{l=1}^{n-1}a_{2l+1}\Big)-1-\sum_{s=1}^{2n}(-1)^s\exp\Big(i\pi t\sum_{k=1}^s a_k\Big),
\end{align*}
whence
\begin{align*}
(I,I)-(J,J)=\frac{2}{\pi^2}\int_0^1 \frac{T(\{a_j\},1-\cos \pi xt)}{t^2}dt.
\end{align*}
The last identity in \eqref{form_eq} follows from the equality above and \eqref{helpful}.
\end{proof}

The important assertion that we are going to prove next states that even if $(I,I)-(J,J)$ can be negative for some $I$ and $J$, it is still greater than a $``$small$"$ negative number. In fact, this is because the value $(I,I)$ is uniformly in $|I|$ close to $\|\chi_I\|_2=|I|$.

\begin{remark}\label{iac} There holds
\begin{align*}
\int_{-1/2}^{1/2}|\widehat{\chi_A}(\xi)|^2d\xi<\int_{-1/2}^{1/2}|\widehat{\chi_A^*}(\xi)|^2d\xi+\frac{4}{\pi^2}.
\end{align*}
\end{remark} 

\begin{proof}
It suffices to show that $(I,I)-(J,J)>-4/\pi^2$. Noting that 
\begin{align*}
(J,J)<\|\chi_J\|_2=:T,
\end{align*}
we obtain
\begin{align*}
(I,I)-(J,J)=2T\int_0^{T/2}\sinc^2 t dt-T=-2T\int_{T/2}^{\infty}\sinc^2 t dt>-\frac{4}{\pi^2},
\end{align*}
and the claim follows.
\end{proof}

We conclude this section with showing that the $L^2$-norms of the non-harmonic trigonometric polynomials that appear in \eqref{form_eq} taken over the whole real line actually coincide. This fact will be eventually very useful for us.

\begin{proposition}\label{equal} For any $n\in\mathbb{N}$ and any real $\alpha_1\leq \alpha_2\leq ...\leq \alpha_{2n}$, there holds
\begin{align*}
\int_0^{\infty}\frac{\Big|\sum_{j=1}^{2n}(-1)^je^{i\alpha_j t}\Big|^2}{t^2}dt = \int_0^{\infty}\frac{\Big|1-e^{i\sum_{j=1}^n(\alpha_{2j}-\alpha_{2j-1})t}\Big|^2}{t^2}dt.
\end{align*}
\end{proposition}

\begin{proof}
Taking into account the last equality in \eqref{form_eq}, we see that
\begin{align*}
\int_0^{\infty}\frac{\Big|\sum_{j=1}^{2n}(-1)^je^{i\alpha_j t}\Big|^2}{t^2}dt &-\int_0^{\infty}\frac{\Big|1-e^{i\sum_{j=1}^n(\alpha_{2j}-\alpha_{2j-1})t}\Big|^2}{t^2}dt=4\int_0^{\infty}\frac{T(\{\alpha_j\},\sin^2(\pi xt/2))}{t^2}dt\\
&=4T(\{\alpha_j\},x/2)\int_0^{\infty}\sinc^2 t dt=0,
\end{align*}
where in the last line we used \eqref{t_zero}.
\end{proof}
Proposition \ref{equal}, together with Corollary \ref{cor} and Propositions \ref{form}, shows the equivalence between Conjectures \ref{conj_1} and \ref{conj_2}.

\section{Proof of Theorem \ref{thm}}\label{sec_pf}
Let us first prove Remark \ref{rem1}, which justifies the appearance of the critical value $T_0$ in the original result of Donoho and Stark. 
\begin{proof}[Proof of Remark \ref{rem1}.] Denote $\delta:=T-T_0$ and let $\varepsilon\in (0,1)$ and $M>0$ to be defined later. By the definition of $T_0$, there exists a finite union of intervals $K, |K|=\delta, K\cap[0,1]=\varnothing,$ and a positive number $w$ such that 
\begin{align*}
\int_{K}\sinc xdx-\int_{T_0}^{T}\sinc xdx=w.
\end{align*}
Consider the function $f=M\chi_{(0,\varepsilon)}+\chi_{(\varepsilon,T_0)}+\chi_{K}=:g_1+g_2+g_3.$ Then $f^*=M\chi_{(0,\varepsilon)}+\chi_{(\varepsilon,T_0)}+\chi_{(T_0,T)}=:g_1+g_2+g_4.$ Let us show that for an appropriate choice of $M$ and $\varepsilon$ we will have
\begin{align}\label{tbp}
0< (f,f)-(f^*,f^*)=2(g_1,g_3)+2(g_2,g_3)-2(g_1,g_4)-2(g_2,g_4).
\end{align}
Note that $(g_2,g_3)-(g_2,g_4)\to (\chi_{(0,T_0)},\chi_{K}-\chi_{(T_0,T)})=:C$ as $\varepsilon\to 0$. At the same time 
\begin{align*}
\frac{(g_1,g_3)-(g_1,g_4)}{\varepsilon}\to M\int_{K}\sinc x dx-M\int_{T_0}^T\sinc x dx=Mw
\end{align*}
as $\varepsilon \to 0$. Thus, chosing $\varepsilon$ sufficiently small and $M$ large enough so that $(g_1,g_3)-(g_1,g_4)+(g_2,g_3)-(g_2,g_4)>Mw\varepsilon/2+C-1>0$, we arrive at \eqref{tbp}.
\end{proof}

Before proceeding to the proof of Theorem \ref{thm}, let us establish the following simple lemma. 

\begin{lemma} Let $r\in\mathbb{N}$ and 
\begin{align*}
0\leq \a_1\leq \beta_1\leq \a_2\leq \beta_2\leq ... \leq \a_r \leq \beta_r\leq 2\pi,\quad\sum_{k=1}^r (\beta_k-\a_k)=\ell.
\end{align*}
Then
\begin{align*}
|S|:=\Big|\sum_{k=1}^r e^{i\beta_k}-e^{i\a_k}\Big|\leq 2\sin\frac{\ell}{2}
\end{align*}
and the equality is attained only when $\beta_1=\a_2,\;...,\;\beta_{r-1}=\a_r$.
\end{lemma}

\begin{proof}
It suffices to proof the following statement: the sum of several clockwise oriented chords corresponding to disjoint arcs of a unit circle of total length $\leq \ell$ cannot exceed $2\sin(\ell/2)=|1-e^{i\ell}|$. This is true, since for any two such consequent chords $\xi_1$ and $\xi_2$ having angle between them $\leq\pi$, either $(\xi_1+\xi_2,S)\leq 0$, so that we can discard both $\xi_1$ and $\xi_2$ and not decrease $|S|$, or $(\xi_1+\xi_2,S)> 0$, which means that merging these chords into one chord $\xi$ parallel to $\xi_1+\xi_2$ increases $|S|$. By doing so several times, we come to the case of just one chord of length $\leq 2\sin(\ell/2)$.
\end{proof}

\begin{proof}[Proof of Theorem \ref{thm}] Assume the contrary: there exist $n> 2$ and $a_1,a_2,...,a_{2n-1}>0$ such that $(I,I)-(J,J)\leq 0$. Take the minimal such $n$ and the corresponding tuple $(a_1,...,a_{2n-1})$ with the minimal $T=\sum_{j=1}^{n}a_{2j-1}\leq 4/3$. 

If $x_{2n}-x_1=:X\leq 2$ (recall the notation \eqref{J}), then by the lemma above we have
\begin{align*}
\int_0^1\frac{\Big|\sum_{j=1}^{2n}(-1)^je^{i\pi x_jt}\Big|^2}{t^2}dt\leq\int_0^1\frac{\Big|1-e^{i\pi \sum_{j=1}^n(x_{2j}-x_{2j-1})t}\Big|^2}{t^2}dt,
\end{align*}
as the absolute value in the first integral is pointwise less than that of the second one, which in light of \eqref{form_eq} yields the result.

Now consider the case $X>2$. Using the notation \eqref{FA_x}, Proposition \ref{p}, and the choice of $J$, we have 
\begin{align*}
F_J(x_1)=F(x_2)=...=F_J(x_{2n})\geq F_I(0)=\Si (T),
\end{align*} 
which yields that 
\begin{align}\label{general}
\int_{J}\sinc x +\sinc (X-x)dx\geq 2\Si (T).
\end{align}
Let us show that \eqref{general} is impossible for any non-decreasing sequence $\{y_j\}_{j=1}^{2n}$ with $y_1=0,\;y_{2n}=X>2,\;\sum_{j=1}^n (y_{2j}-y_{2j-1})=T$, unless all the terms in the latter sum except for one are zero. Indeed, assume the contrary and take a sequence $\{y_j\}$ maximizing the left-hand side of \eqref{general} and define $J'$ to be the union of the intervals generated by $\{y_j\}$ in the same way as $J$ is generated by $\{x_j\}$. If there is an index $j$ such that $y_j=y_{j+1}$, we can delete both points and keep doing so until $\{y_j\}_{j=1}^m,\;m\leq n,$ is increasing. By the assumption, $m>1$. 

Assume that there is a point $y_{2k}\in (0,2/3)\setminus J',\;k\geq 1$, whence $M_1:=(0,2/3)\setminus J'\neq\varnothing$. If $M_2:=J'\setminus ((0,\pi]\cup[x_{2n}-\pi,x_{2n}))\neq \varnothing$, then for sufficiently small $\delta>0$ we can remove a subinterval of $J'$ of size $\delta$ in $M_2$ and substitute it by $[y_{2k},y_{2k}+\delta]$, which will increase 
\begin{align*}
\int_{J'}\sinc y +\sinc (X-y)dy,
\end{align*}
since $\sinc y+\sinc (X-y)>\sinc z+\sinc (X-z)$ for $y\in M_1$ and $z\in M_2$. The latter is true due to the following observations. First, if $X\leq 3$, then $\sinc z,\;\sinc (X-z)<0$ and $\sinc y+\sinc (X-y)\geq \sinc (2/3)+\min_t \sinc t>0.45-0.22>0$. Second, if $X>3$, then
\begin{align*}
\sinc (2/3)>0.4>2\cdot 0.13+0.1>2\max_{x\geq 1}\sinc x-\min_{x\geq 7/3}\sinc x.
\end{align*}
The contradiction with the optimality of $\{y_j\}$ shows that $M_2=\varnothing$. Further, since 
$$\max_{4/11\leq x\leq 1}(\sinc x)'=\min_{1\leq x}(\sinc x)',$$
we deduce that $y_{2k}<4/11$. However, in this case
\begin{align*}
\sinc y_{2k} +\sinc (X-y_{2k})dy &\geq\sinc (4/11)+\min_{x\geq 1}\sinc x>0.79-0.22>0.42+0.13\\
&>\sinc(2/3)+\max_{x\geq 1}\sinc x>\sinc z +\sinc (X-z)dy
\end{align*}
for any $z\notin J'\cup [0,2/3]\cup [X-2/3,X]$, which once again contradicts the optimality of $\{y_j\}$. Thus, $J'=[0,T/2]\cup [X-T/2,X]$. It remains to observe that $\max_{x\geq y}(\Si (x+y)-\Si(x))=\Si(2y)-\Si(y)$ for $0.88<T_0\leq 2y\leq 4/3$, in order to see that \eqref{general} cannot hold. Indeed, if $x\leq T_0,$ this simply follows from the definition of $T_0$. Otherwise, for $x\geq \max(y,T_0)$,
\begin{align*}
\Si(2y)-\Si(y)&-(\Si(x+y)-\Si(x))>\Si(0.88)-\Si(y)-(2y-0.88)\sinc T_0-y\sinc T_0\\
&>\Si(0.88)-\Si(2/3)-(2-0.88)\sinc T_0>0.18-1.12\cdot 0.14>0,
\end{align*}
This concludes the proof of the theorem.
\end{proof}

\section{Characteristic function of the union of two intervals}\label{sec_two}
The first step to be done towards Conjecture \ref{conj_1} is to prove it under a very restricting additional assumption that the set $A$ is the union of two intervals. We will see that in this case one must indeed bring the intervals together in order to increase the quadratic spectral concentration of the characteristic function. 
\begin{theorem}\label{prop} Among all pairs of non-overlapping intervals $I_1$ and $I_2$ with $|I_1|+|I_2|=T>0$ fixed, the maximum of $\int_{-W/2}^{W/2}|\widehat{\chi_{I_1\cup I_2}}(\xi)|^2d\xi$ is attained when $I_1\cup I_2$ is an interval, i.e. 
\begin{align*}
\int_{-W/2}^{W/2}|\widehat{\chi_{I_1\cup I_2}}(\xi)|^2d\xi\leq \int_{-W/2}^{W/2}|\widehat{\chi_{I_1\cup I_2}^*}(\xi)|^2d\xi.
\end{align*}
\end{theorem}

\begin{proof}
Denote $I_1=:[0,a],\;I_2=:[-b-\varepsilon,-\varepsilon]$ for some $\varepsilon\geq 0$. By Lemma \ref{prop_1}, it suffices to show the positivity of 
\begin{align*}
\Phi(a,b,\varepsilon)=\int_0^1t^{-2}\sin\frac{\pi at}{2}\sin\frac{\pi bt}{2}\sin\frac{\pi \varepsilon t}{2}\sin\frac{\pi(a+b+\varepsilon)t}{2}dt.
\end{align*} 
Assume the contrary: $\Phi(a,b,\varepsilon)\leq 0$ for some $a,b,\varepsilon$. Then defining $R(s):=\Phi(as,bs,\varepsilon s)/s$, we have that
\begin{align*}
R(s)=\int_0^st^{-2}\sin\frac{\pi a t}{2}\sin\frac{\pi bt}{2}\sin\frac{\pi\varepsilon t}{2}\sin\frac{\pi(a+b+\varepsilon)t}{2}dt
\end{align*}
assumes non-positive values. A local minimum of $R(s)$ in $s$ is attained only when one of the values $a,b,\varepsilon,a+b+\varepsilon$ is a multiple of $2$. So we have two cases:

\textbf{Case 1. One of the values ${\bf {a,b,\boldsymbol\varepsilon}}$ is a multiple of ${\bf 2}$.} Due to the symmetry in $a,b,\varepsilon$ we can assume without loss of generality that $\varepsilon=2k,\;k\in\mathbb{N},\;a\geq b$. We have
\begin{align*}
\int_0^a\int_0^b\frac{\sin\pi(x+y)}{\pi(x+y)}dxdy&-\int_0^a\int_0^b\frac{\sin\pi(x+y+2k)}{\pi(x+y+2k)}dxdy\\
&=\frac{2k}{\pi}\int_0^a\int_0^b \frac{\sin\pi(x+y)}{(x+y)(x+y+2k)}dxdy\\
&=\frac{2k}{\pi}\int_0^{a+b}\frac{\sin \pi t}{t(t+2k)}\int_{\max(0,t-b)}^{\min(a,t)}1 ds dt.
\end{align*}
The positivity of the last integral follows from the monotonicity of the function
\begin{align*}
\frac{\min(a,t)-\max(0,t-b)}{t(t+2k)}=\begin{cases}
\frac{1}{t+2k},&t\in [0,b],\\
\frac{b}{t(t+2k)},&t\in [b,a],\\
\frac{a+b-t}{t(t+2k)},&t\in [a,a+b],\\
0,&t>a+b.
\end{cases}
\end{align*}

\textbf{Case 2. ${\bf a+b+\boldsymbol{\varepsilon}=2k,\;k\in\boldsymbol{\mathbb{N}}}$.} In this case
\begin{align}\label{second_case}
\int_0^a\int_{-b}^0\frac{\sin\pi(x-y)}{\pi(x-y)}dxdy&-\int_0^a\int_{2k-b}^{2k}\frac{\sin\pi(x-y)}{\pi(x-y)}dx dy\nonumber\\
&=\frac{1}{\pi}\int_0^a\int_{-b}^0 \Big(\frac{1}{x-y}-\frac{1}{x-y-2k}\Big)\sin\pi(x-y)dxdy\nonumber\\
&=\frac{2k}{\pi}\int_0^a\int_0^b\frac{\sin\pi(x+y)}{(2k-x-y)(x+y)}dx dy\nonumber\\
&=\frac{2k}{\pi}\int_0^{k}\frac{\sin \pi t}{t(2k-t)}\Big(\int_{\max(0,t-b)}^{\min(a,t)}1-\int_{\max(0,2k-t-b)}^{\min(a,2k-t)}\Big) ds dt\nonumber\\
=\frac{2k}{\pi}\int_0^{k}\frac{\sin \pi t}{t(2k-t)}(\min(a,t)&+\max(0,2k-t-b)-\max(0,t-b)-\min(a,2k-t))dt.
\end{align}
Note that $t\leq k$ and without loss of generality $b\leq a\leq k$, whence $\min(a,2k-t)=a$ and $\max(0,2k-t-b)=2k-t-b$, so that the last integral in \eqref{second_case} becomes
\begin{align*}
\int_0^{k}\frac{\sin \pi t}{t(2k-t)}(\min(a,t)-\max(0,t-b)+2k-t-b-a)dt.
\end{align*} 

{Case a. $t<b\leq a$.} Then $\min(a,t)-\max(0,t-b)+2k-t-b-a=l-b-a$ and $(2k-b-a)/(t(2k-t))$ is decreasing on $(0,b)$.

{Case b. $b\leq t< a$.} We have $\min(a,t)-\max(0,t-b)+2k-t-b-a=2k-t-a$ and 
\begin{align*}
\Big(\frac{2k-t-a}{(2k-t)t}\Big)'=\frac{4kt-t^2-(2k)^2-2at+2ak}{(2k-t)^2t^2}<-\frac{(a+t-2k)^2}{(2k-t)^2t^2}\leq 0.
\end{align*}

{Case c. $b\leq a\leq t$.} Then $\min(a,t)-\max(0,t-b)+2k-t-b-a=2k-2t$ and $(2k-2t)/((2k-t)t)$ is decreasing on $(a,k)$.

Thus, we showed that the sine function is integrating against a decreasing function in \eqref{second_case}, whence we see that at any local minimum of $R(s)$ there holds $R(s)>0$, and the needed is proved.
\end{proof}

\begin{corollary} For any real $\alpha_1\leq \alpha_2\leq \a_3\leq \a_4$, one has
\begin{align*}
\int_1^{\infty}\frac{|e^{i\a_1 t}-e^{i\a_2 t}+e^{ia_3 t}-e^{i\a_4 t}|^2}{t^2}dt\geq \int_1^{\infty}\frac{|1-e^{i(\a_1-\a_2+\a_3-\a_4)t}|^2}{t^2}dt.
\end{align*}
\end{corollary}

\section{A couple of particular cases}\label{sec_particular}
An intuition behind the conviction that Conjecture \ref{conj_1} must be true is the following. It seems that it is better to integrate the function $\sinc^2 x$ over an interval starting at the origin rather than over several disjoint intervals of the same total size because of the significance of the factor $1/x^2$ that appears in $\sinc^2 x$. Although in general this intuition does not provide a straightforward argument, it does work in case of the lengths of all intervals being even positive integers.

\begin{proposition}\label{multiple}
If all $a_j$ are multiples of $2$, then $(I,I)>(J,J)$.
\end{proposition}

\begin{proof}
Denote $a_j=:2k_j$ with $k_j\in\mathbb{N}$ for all $j=1,2,...,2n-1$, and let $K_j:=\sum_{m=1}^j k_m$. We have
\begin{align*}
\frac{1}{2}((I,I)-(J,J))&=T\Big(\{a_j\},x\int\limits_0^{x/2}\sinc^2 t dt\Big)\\
&=\sum_{j=1}^{n}a_{2j-1}\Big(\int\limits_{0}^{\sum_{m=1}^nk_{2m-1}}\sinc^2 t dt-\sum_{s=0}^{2j-2}(-1)^s\sum_{t=j}^n\int\limits_{k_{2t-2}}^{k_{2t-1}}\sinc^2(t-k_{s}) dt\Big)\\
&+\sum_{j=1}^{n-1}a_{2j}\sum_{s=0}^{2j-1}(-1)^{s}\sum_{t=j+1}^n\int\limits_{k_{2t-2}}^{k_{2t-1}}\sinc^2(t-k_{s}) dt\\
&=\sum_{j=1}^{n}a_{2j-1}\Big(\int\limits_{0}^{\sum_{m=1}^n k_{2m-1}}\sinc^2 t dt-\sum_{s=0}^{2j-2}(-1)^s\sum_{t=j}^n\int\limits_{K_{2t-2}-K_s}^{K_{2t-1}-K_s}\sinc^2 t dt\Big)\\
&+\sum_{j=1}^{n-1}a_{2j}\sum_{s=0}^{2j-1}(-1)^{s}\sum_{t=j+1}^n\int\limits_{K_{2t-2}-K_s}^{K_{2t-1}-K_s}\sinc^2 t dt>0,
\end{align*}
since $\int_{m}^{m+1}\sinc^2 t dt$ is decreasing in $m\in\mathbb{Z}_+$.
\end{proof}

Another particular case (in fact, a very particular one) for which we can ensure that $(I,I)>(J,J)$ once Theorem \ref{prop} is established, corresponds to $n=3$.

\begin{remark}\label{rem_for_three} If $n=3$ and $a_2=a_5$, then $(I,I)>(J,J)$.
\end{remark}

\begin{proof}
Denote the three intervals in question by $I_1,I_2$, and $I_3$ in the natural order and let the interval between $I_1$ and $I_2$ be $I'$. In light of the equality $|I_3|=|I'|$ we can move $I_3$ to $I'$ covering thereby the hole between $I_1$ and $I_3$. Then by Theorem \ref{prop} we have $(I_1,I')>(I_1,I_3)$ and $(I_2,I')>(I_2,I_3)$, whence $(I,I)>(J,J)$, as desired.
\end{proof}

For the sake of completeness, we summarize the particular cases in which we established the validity of Conjecture \ref{conj_2} in the form of 

\begin{corollary} For $n\in\mathbb{N}$ and $\alpha_1\leq \alpha_2\leq ...\leq \alpha_{2n}$, there holds
\begin{align*}
\int_1^{\infty}\frac{\Big|\sum_{j=1}^{2n}(-1)^je^{i\alpha_j t}\Big|^2}{t^2}dt\geq \int_1^{\infty}\frac{\Big|1-e^{i\sum_{j=1}^n(\alpha_{2j}-\alpha_{2j-1})t}\Big|^2}{t^2}dt
\end{align*}
in the following cases:

1) $\sum_{j=1}^n(\a_{2j}-\a_{2j-1})\leq \frac{4\pi}{3};$

2) $n=2$;

3) $\a_j=2\pi k_j,\;k_j\in\mathbb{N},\;j=1,2,...,2n-1;$

4) $n=3,\;\a_3-\a_2=\a_6-\a_5$.
\end{corollary}

\begin{proof}
The first two parts correspond to Theorems \ref{thm} and \ref{prop}, respectively, while the last two ones are proved in Proposition \ref{multiple} and Remark \ref{rem_for_three}. 
\end{proof}

\section{Estimates for the $L^2$-norms of non-harmonic polynomials}\label{sec_l2}
In this section, with the help of the tools we developed up to now, we derive estimates for the $L^2$-norms of non-harmonic trigonometric polynomials with alternating coefficients $\pm 1$, i.e. polynomials of the form $\sum_{j=1}^{2n}(-1)^je^{i \a_jt},\;\a_1\leq\a_2\leq...\leq \a_{2n}$. Our aim is to prove Theorem \ref{l2}: for the sake of convenience, we split it into two propositions. 

\begin{proposition}\label{l21} For any positive integer $n$ and any real $\alpha_1\leq \alpha_2\leq ...\leq \alpha_{2n}$,
\begin{align*}
\int_{-1/2}^{1/2}\Big|\sum_{j=1}^{2n}(-1)^je^{i \a_jt}\Big|^2dt<\frac{\pi}{2} \sum_{j=1}^n (\a_{2j}-\a_{2j-1}).
\end{align*}
\end{proposition}
Note that a trivial bound for such a polynomial would give quadratic order in $\sum_{j=1}^n (\a_{2j}-\a_{2j-1})$ and not linear, as it appears in the bound above.

\begin{proof}
According to Proposition \ref{equal},
\begin{align*}
\int_{-1/2}^{1/2}\Big|\sum_{j=1}^{2n}(-1)^je^{i \a_jt}\Big|^2dt&=\int_{0}^{1}\Big|\sum_{j=1}^{2n}(-1)^je^{\frac{i \a_jt}{2}}\Big|^2dt<\int_{0}^{\infty}\frac{\Big|\sum_{j=1}^{2n}(-1)^je^{\frac{i \a_jt}{2}}\Big|^2}{t^2}dt\\
&=\int_{0}^{\infty}\frac{\Big|1-e^{i\sum_{j=1}^n(\a_{2j}-\a_{2j-1})\frac{t}{2}}\Big|^2}{t^2}dt=4\int_{0}^{\infty}\frac{\sin^2\frac{\sum_{j=1}^n(\a_{2j}-\a_{2j-1})t}{4}}{t^2}dt\\
&=\frac{\pi}{2} \sum_{j=1}^n(\a_{2j}-\a_{2j-1}).
\end{align*}
\end{proof}
For convenience, from now on denote
\begin{align}\label{H_def}
H(a_1,...,a_{2n-1}):=(I,I)-(J,J).
\end{align}
In order to obtain the other bound appearing in Theorem \ref{l2}, we need a couple of auxiliary lemmas that will be formulated in terms of the spectral concentration of characteristic functions. The first lemma states that the average of $H(sa_1,...,sa_{2n-1})$ over all $s\in[0,1]$ is bounded above uniformly in $a_1,...,a_{2n-1}$, so that the average of $(I',I')-(J',J')$ over all $``$shrinkings$"$ $I'$ and $J'$ of $I$ and $J$ is bounded above independently of $|I|$. To put it in a rigorous way, for a set $X$ and $s>0$, denote $s X:=\{s x:\;x\in X\}$.

\begin{lemma} There holds
\begin{align*}
\int_0^1 (s I,s I)-(s J,s J)\;ds< \frac{27(n-1)}{\pi^2}.
\end{align*}
\end{lemma}

In fact, we will see later (in Corollary \ref{cor_new}) that this bound implies an analogous one for the difference $(I,I)-(J,J)$ itself.

\begin{proof}
Denote by $J_1,...,J_n$ the intervals that generate $J$ and divide $I$ into intervals $I_1,...,I_n$ so that $|I_k|=|J_k|$ for all $k=1,...,n$ (both $I_k$'s and $J_k$'s are enumerated from the left to the right). Then 
\begin{align*}
\int_0^1 (s I,s I)-(s J,s J)\;ds=2\sum_{k<m}\int_0^1(s I_k,s I_m)ds-2\sum_{k<m}\int_0^1(s J_k,s J_m)ds.
\end{align*}
Note that 
\begin{align*}
|\{x\in J:\;x\in J_k,\exists y\in J_m,m>k,x+y\leq 1\}|\leq n-1,
\end{align*}
whence
\begin{align}\label{nt}
\sum_{k<m}\int_0^1(s J_k,s J_m)ds>-(n-1)+\sum_{k<m}\int_0^1\int_{J_k}\int_{J_m}\chi_{\{x-y\geq 1\}}\frac{s\sin(x-y)s}{(x-y)}dxdyds.
\end{align}
Next we have
\begin{align*}
&\int_0^1\int_{J_k}\int_{J_m}\chi_{\{x-y\geq 1\}}\frac{s\sin(x-y)s}{(x-y)}dxdyds\\
&=\int_{J_k}\int_{\cup_{m>k}J_m}\frac{\chi_{\{x-y\geq 1\}}}{\pi^3(x-y)^3}\int_0^{\pi(x-y)}s\sin s \;ds dx dy\\
&=\int_{J_k}\int_{\cup_{m>k}J_m}\chi_{\{x-y\geq 1\}}\frac{\sin \pi(x-y)-\pi(x-y)\cos\pi (x-y)}{\pi^3(x-y)^3}dx dy\\
&>-\frac{1}{2\pi^3}-\int_{\cup_{m>k}J_m}\int_{J_k}\chi_{\{x-y\geq 1\}}\frac{\cos\pi (x-y)}{\pi^2(x-y)^2}dy dx>-\frac{1}{2\pi^3}-\int_{1}^{\infty}\frac{1}{\pi^2x^2}dx,
\end{align*}
which along with \eqref{nt} implies
\begin{align*}
\int_0^1 (s I,s I)-(s J,s J)\;ds<2\sum_{k=1}^{n-1}\int_0^1\Big(s I_k,\bigcup_{m=k+1}^n sI_m\Big)ds+\frac{23(n-1)}{\pi^2},
\end{align*}
where we used that $2\pi^2(1+1/2\pi^3+1/\pi^2)<23$. To conclude the proof, it suffices to show that for two non-overlapping intervals $I'$ and $I''$ sharing the same endpoint, there holds $(I',I'')\leq 2/\pi^2$. This can be seen through the proof of Theorem \ref{prop}, as it is shown there that $(I',I'')$ is an integral of sine against a monotone function which is $1/\pi^2$ at the origin. 
\end{proof}

The next statement provides a connection between the value $(I,I)-(J,J)$ and the average $\int_0^1(s I,s I)-(s J,s J)ds$.

\begin{lemma} There holds
\begin{align*}
\int_0^1(s I,s I)-(s J,s J)\;ds=\frac{(I,I)-(J,J)}{2}-\frac{2}{\pi^2}T\Big(\{a_j\},\int_0^1\sin^2\frac{\pi s x}{2}ds\Big).
\end{align*}
\end{lemma}

\begin{proof}
First of all, recall that
\begin{align*}
\frac{x}{2}\int_0^{x/2}\sinc^2 t dt=-\frac{1}{\pi^2}\sin^2\frac{\pi x}{2}+\frac{x}{2}\Si(x),
\end{align*}
whence by Proposition \ref{form},
\begin{align*}
H(a_1,...,a_{2n-1})=-\frac{4}{\pi^2}T(\{a_j\},\sin^2(\pi x/2))+2T(\{a_j\},x\Si(x)).
\end{align*}
The claim follows then in light of the equalities
\begin{align*}
\int_0^1 2s x\Si(s x)ds=\Si(x)-\frac{\sinc x}{\pi^2}+\frac{\cos \pi x}{\pi^2}
\end{align*}
and
\begin{align*}
\int_0^1-\frac{4}{\pi^2}\sin^2\frac{\pi sx}{2}ds=-\frac{2}{\pi^2}+\frac{2\sinc x}{\pi^2}.
\end{align*}
\end{proof}

Now, invoking the lemmas above, we can establish the following estimate, which is, as mentioned before, sharp up to a constant.
\begin{proposition}\label{l22} For any positive integer $n$ and any real $\alpha_1\leq \alpha_2\leq ...\leq \alpha_{2n}$,
\begin{align*}
\int_{-1/2}^{1/2}\Big|\sum_{j=1}^{2n}(-1)^je^{i \a_jt}\Big|^2dt<54n-47.
\end{align*}
\end{proposition}

\begin{proof}
Combining the lemmas above and Remark \ref{iac}, we get
\begin{align*}
\frac{27(n-1)}{\pi^2}&>\int_0^1(s I,s I)-(s J,s J)\;ds\\
&=\frac{(I,I)-(J,J)}{2}-\frac{2}{\pi^2}T\Big(\{a_j\},\int_0^1\sin^2\frac{\pi s x}{2}ds\Big)\\
&>-\frac{2}{\pi^2}+\frac{n-1}{\pi^2}+\frac{1}{\pi^2}T(\{a_j\},\sinc x)\\
&=-\frac{3}{\pi^2}+\frac{\sinc T}{\pi^2}+\frac{1}{2\pi^2}\int_{-1/2}^{1/2}\Big|\sum_{j=1}^{2n}(-1)^je^{2\pi i x_j t}\Big|^2dt.
\end{align*}
Hence,
\begin{align*}
\int_{-1/2}^{1/2}\Big|\sum_{j=1}^{2n}(-1)^je^{2\pi i x_j t}\Big|^2dt<54(n-1)+6-2\min_{x}\sinc x<54n-47,
\end{align*}
for all $x_1\leq x_2\leq ...\leq x_{2n}$, which concludes the proof of the proposition and therefore establishes Theorem \ref{l2}.
\end{proof}

Finally, we are going to show that, whatever the intervals in $J$ are, they provide a uniform in $a_1,...,a_{2n-1}$ $``$approximation$"$ to $(I,I)$ by $(J,J)$. 
\begin{corollary}\label{cor_new}
If $A$ is the union of $n$ intervals, then
\begin{align*}
\int_{-1/2}^{1/2}|\widehat{\chi_A^*}(\xi)|^2d\xi-\int_{-1/2}^{1/2}|\widehat{\chi_A}(\xi)|^2d\xi<\frac{54n-51}{\pi^2}.
\end{align*}
\end{corollary}

\begin{proof}
Under our notation, the statement is equivalent to the inequality $(I,I)-(J,J)<(54n-51)/\pi^2$. The latter can be shown following the proof of Proposition \ref{l22}:
\begin{align*}
(I,I)-(J,J)<\frac{54(n-1)}{\pi^2}-\frac{2\sinc T-2}{\pi^2}-\frac{1}{\pi^2}\int_{-1/2}^{1/2}\Big|\sum_{j=1}^{2n}(-1)^je^{2\pi i x_j t}\Big|^2dt<\frac{54n-51}{\pi^2}.
\end{align*}
\end{proof}

\section{Further observations}\label{sec_obs}

To begin this section, we will see that if for some $n>1$ and some $a_1,a_2,...,a_{2n-1}>0$ there holds $(I,I)\leq (J,J)$, then for the minimal such $n$, there exists a tuple $(a_1,a_2,...,a_{2n-1})$ of positive numbers delivering a local minimum of $(I,I)-(J,J)$. This can be seen via the following lemma.

\begin{lemma}\label{infinito} Let $n$ be the minimal such that for some $a_1,a_2,...,a_{2n-1}>0$ there holds $(I,I)\leq (J,J)$. Then, for any $j=1,2,...,2n-1$,
\begin{align*}
\lim_{a_j\to\infty}((I,I)-(J,J))>0.
\end{align*}
\end{lemma}

\begin{proof}
Under notation \eqref{H_def}, the minimality of $n$ yields
\begin{align*}
\lim_{a_{2k}\to\infty}H(a_1,a_2,...,a_{2n-1})&=H(a_1,...,a_{2k-1})+H(a_{2k+1},...,a_{2n-1})\\
&+2\Big(\Big[-\sum_{m=1}^{k}a_{2m-1},0\Big],\Big[0,\sum_{m=k+1}^{n}a_{2m-1}\Big]\Big)>0.
\end{align*}
For $a_{2k-1}$, we first observe that
\begin{align*}
\Big|a_{2k-1}\sum_{s\leq 2k-1\leq t}(-1)^{t-s+1}\int_{a_s+...+a_t}^{\infty}\sinc^2 t\Big|\leq 2na_{2k-1}\sum_{j=1}^{n-1}a_{2j-1}\frac{1}{\pi^2a_{2k-1}^2}\to 0,
\end{align*}
which along with the first identity in \eqref{t_j_zero} for $j=2k-1$ implies that $T_{2k-1}(\{a_j\},G(x/2))\to 0$ as $a_{2k-1}\to\infty$. Thus,
\begin{align*}
\lim_{a_{2k-1}\to\infty}&H(a_1,a_2,...,a_{2n-1})=H(a_1,...,a_{2k-3})+H(a_{2k+1},...,a_{2n-1})\\
+&4\sum_{s=0}^{2n-2k-1}(-1)^{s+1}G\Big(\frac{a_{2k}+...+a_{2k+s}}{2}\Big)+4\sum_{s=0}^{2k-3}(-1)^{s+1}G\Big(\frac{a_{2k-2}+...+a_{2k-2-s}}{2}\Big)>0
\end{align*}
due to the monotonicity of the function $G$.
\end{proof}

By Lemma \ref{infinito}, there must exist a non-negative local minimum of $H$ with $a_1,a_2,...,a_{2n-1}>0$. At such a point we have $H'_{a_m}(a_1,...,a_{2n-1})=0$ for all $m=1,...,2n-1$, which is the same as having the equalities $F_J(x_j)=F_I(0)$ for all $j=1,2,...,2n$ (recall \eqref{FA_x}). Moreover, there must hold 
\begin{align*}
(-1)^j F'_J(x_j)\geq F'_I(0)
\end{align*}
for all $j=1,2,...,2n$, which means that
\begin{align}\label{second_der}
\sum_{s=0}^{2n-1}(-1)^{s-j}\sinc (x_s-x_j)\geq 1-\sinc T
\end{align}
for all $j=1,2,...,2n.$ Summing \eqref{second_der} over all $j=1,2,...,2n$, we obtain that
\begin{align*}
2n+2T(\{a_j\},\sinc)\geq 2n-2(n-1)\sinc(T),
\end{align*}
which is equivalent to
\begin{align}\label{impo}
\int_{0}^{1}\Big|\sum_{j=1}^{2n}(-1)^je^{i\pi x_jt}\Big|^2dt=\int_{-1/2}^{1/2}\Big|\sum_{j=1}^{2n}(-1)^je^{2i\pi x_jt}\Big|^2dt\geq 2n-2(n-1)\sinc T.
\end{align}
In light of Proposition \ref{l2}, inequality \eqref{impo} yields
\begin{align*}
2n-2(n-1)\sinc T<\pi^2 T=\pi^2\sum_{j=1}^{n}a_{2j-1}.
\end{align*}
Thus, any non-positive local minimum of $H$ must be located quite $``$far away$"$ from the origin.

Another useful observation that can be made regarding the aforementioned point of local minimum is that there must hold
\begin{align*}
H(sa_1,...,sa_{2n-1})'_s\big|_{s=1}=0,
\end{align*} 
so that
\begin{align*}
0&=\frac{2}{\pi^2}\Big(s\int_0^s \frac{T(\{a_j\},1-\cos \pi xt)}{t^2}dt\Big)'_s\Big|_{s=1}\\
&=\frac{2}{\pi^2}\Big(\int_0^s \frac{T(\{a_j\},1-\cos \pi xt)}{t^2}dt+\frac{T(\{a_j\},1-\cos \pi xs)}{s}\Big)\Big|_{s=1}\\
&=H(a_1,...,a_{2n-1})+\frac{2}{\pi^2}T(\{a_j\},1-\cos \pi x)\\
&=H(a_1,...,a_{2n-1})-\frac{1}{\pi^2}\Big|\sum_{j=1}^{2n}(-1)^je^{i\pi x_j}\Big|^2+\frac{1}{\pi^2}\Big|1-e^{i\pi \sum_{j=1}^n(x_{2j}-x_{2j-1})}\Big|^2.
\end{align*}
The chain of equalities above in turn implies that
\begin{align*}
(I,I)-(J,J)&=H(a_1,a_2,...,a_{2n-1})\\
&=\frac{1}{\pi^2}\Big|\sum_{j=1}^{2n}(-1)^je^{i\pi x_j}\Big|^2-\frac{1}{\pi^2}\Big|1-e^{i\pi \sum_{j=1}^n(x_{2j}-x_{2j-1})}\Big|^2\\
&\geq -\frac{1}{\pi^2}\Big|1-e^{i\pi \sum_{j=1}^n(x_{2j}-x_{2j-1})}\Big|^2
\end{align*}
and
\begin{align*}
\Big|\sum_{j=1}^{2n}(-1)^je^{i\pi x_j}\Big|^2\leq\Big|1-e^{i\pi \sum_{j=1}^n(x_{2j}-x_{2j-1})}\Big|^2=2-2\cos \pi T.
\end{align*}
This means that at a point of local minimum of $H$ the value of $\Big|\sum_{j=1}^{2n}(-1)^je^{i\pi x_j}\Big|^2$ is small, while the integral $\int_{0}^{1}\Big|\sum_{j=1}^{2n}(-1)^je^{i\pi x_jt}\Big|^2dt$ is large (in fact, it must be up to a constant maximal possible due to \eqref{impo} and Proposition \ref{l22}).


\begin{acknowledgements} I am grateful to Miquel Saucedo for inspiring discussions of related problems.
\end{acknowledgements}


\end{document}